\documentclass[11pt]{amsart}
\usepackage{mathpple}

\usepackage{upgreek}
\usepackage[mathscr]{euscript}

\usepackage{stmaryrd}
\usepackage{amssymb}
\usepackage{amsmath}
\usepackage{titletoc}
\usepackage{extarrows}
\usepackage{varioref}
\usepackage{bm}
\usepackage{amsthm}
\usepackage[all,cmtip]{xy}
\usepackage{tikz-cd}
\usepackage{color} 
\usepackage{soul}
\setulcolor{red}
\sethlcolor{yellow}
\setstcolor{blue}
\usepackage{amscd}
\usepackage{xfrac}    

\usepackage{amsfonts}
\usepackage[colorlinks,linkcolor=blue,citecolor=blue, pdfstartview=FitH]{hyperref}
\usepackage{graphicx}
\usepackage{geometry}
\usepackage{cite}

\newtheorem{theorem}{Theorem}[section]
\newtheorem{corollary}{Corollary}
\newtheorem{lemma}{Lemma}
\newtheorem{proposition}{Proposition}
\theoremstyle{definition}
\newtheorem{example}{Example}
\newtheorem{definition}{Definition}

\newtheorem{def-thm}{Definition-Theorem}
\newtheorem{def-prop}{Definition-Proposition}

\theoremstyle{remark}
\newtheorem{remark}{Remark}
\numberwithin{equation}{subsection}
\numberwithin{theorem}{subsection}
\numberwithin{definition}{subsection}
\numberwithin{corollary}{subsection}
\numberwithin{proposition}{subsection}
\numberwithin{lemma}{subsection}

\newcommand{\abs}[1]{\left\vert#1\right\vert}

\newcommand{\sE}{{\mathcal E}}
\newcommand{\sF}{{\mathcal F}}

\newcommand{\sH}{{\mathcal H}}

\newcommand{\sL}{{\mathcal L}}
\newcommand{\sM}{{\mathcal M}}
\newcommand{\sN}{{\mathcal N}}
\newcommand{\sO}{{\mathcal O}}
\newcommand{\sP}{{\mathcal P}}

\newcommand{\sS}{{\mathcal S}}
\newcommand{\sT}{{\mathcal T}}
\newcommand{\sU}{{\mathcal U}}

\newcommand{\sW}{{\mathcal W}}
\newcommand{\sX}{{\mathcal X}}

\newcommand{\BC}{\mathbb C}

\newcommand{\BF}{\mathbb F}

\newcommand{\BP}{\mathbb P}
\newcommand{\BQ}{\mathbb Q}

\newcommand{\BZ}{\mathbb Z}

\newcommand{\Br}{\operatorname{Br}}

\newcommand{\Pic}{\operatorname{Pic}}

\newcommand{\Fr}{\operatorname{Fr}}

\newcommand{\Jac}{{\operatorname{Jac}}}

\newcommand{\Prym}{\operatorname{Prym}}

\newcommand{\GL}{\operatorname{GL}}
\newcommand{\PGL}{\operatorname{PGL}}
\newcommand{\SL}{\operatorname{SL}}

\newcommand{\End}{\operatorname{End}}

\newcommand{\Tr}{\operatorname{Tr}}
\newcommand{\Nm}{\operatorname{Nm}}

\newcommand{\Spec}{\operatorname{Spec}}

\newcommand{\inv}{\operatorname{inv}}
\newcommand{\Split}{\operatorname{Split}}

\newcommand{\Aut}{\operatorname{Aut}}

\newcommand{\isom}{\operatorname{isom}}

\newcommand{\codim}{{\operatorname{codim}}}

\newcommand{\gr}{{\operatorname{gr}}}

\newcommand{\pdeg}{\text{par-}deg}

\newcommand{\pt}{{\scriptscriptstyle\bullet}}
\newcommand{\btheorem}{\begin{theorem}}
	\newcommand{\etheorem}{\end{theorem}}
\newcommand{\bproposition}{\begin{proposition}}
	\newcommand{\eproposition}{\end{proposition}}
\newcommand{\bdefinition}{\begin{definition}}
	\newcommand{\edefinition}{\end{definition}}
\newcommand{\bcorollary}{\begin{corollary}}
	\newcommand{\ecorollary}{\end{corollary}}
\newcommand{\bproof}{\begin{proof}}
	\newcommand{\eproof}{\end{proof}}
\newcommand{\bremark}{\begin{remark}}
	\newcommand{\eremark}{\end{remark}}
\newcommand{\eexample}{\end{example}}
\newcommand{\bexample}{\begin{example}}
\newcommand{\elemma}{\end{lemma}}
\newcommand{\blemma}{\begin{lemma}}

\newcommand{\bsf}[1]{\textbf{\textsf{#1}}}
\renewcommand{\bf}[1]{\textbf{#1}}
\newcommand{\tx}[1]{\text{#1}}

\begin{document}
	
	\title[]{Topological Mirror Symmetry of Parabolic Hitchin Systems}%
	\author{Xiaoyu Su${}^1$, Bin Wang${}^2$, Xueqing Wen${}^3$}%
	\address{{Address ${}^1$: Yau Mathematical Science Center, Beijing, 100084, China.}}
	\email{\href{mailto:email address}{{suxiaoyu@mail.tsinghua.edu.cn}}}
	\address{{Address ${}^2$: Steklov Mathematical Institute of Russian Academy of Sciences, Moscow, 119991, Russia.}}
	\email{\href{mailto:email address}{{binwang@mi-ras.ru}}}
	\address{{Address ${}^3$: Yau Mathematical Science Center, Beijing, 100084, China.}}
	\email{\href{mailto:email address}{{xueqingwen@mail.tsinghua.edu.cn}}}
	
	
	\begin{abstract}
		In this paper, we first prove the parabolic Beauvile-Narasimhan-Ramanan correspondence over an arbitrary field which generalizes the corresponding results over algebraically closed fields in \cite{SWW22}. We use the correspondence and the p-adic integration methods developed by Groechenig-Wyss-Ziegler \cite{GWZ20m} to prove the topological mirror symmetry for parabolic Hitchin systems on curves with arbitrary parabolic structures. 
	\end{abstract}
	\maketitle
	
	\section{Introduction}
	The theory of Mirror Symmetry stems from string theory which predicts a mysterious interchanging of symplectic geometry and complex geometry between two Calabi-Yau manifolds induced from different types of string theories, called mirror pairs. Since then, many mathematic works have devoted efforts into making ``symplectic-complex transformation" precise in mathematics and finding examples of mirror pairs.
	
	One of the approaches is topological mirror symmetry. It is due to the following observation. Let $(X,Y)$ be a mirror pair satisfying the (vague) symplectic-complex transformation. From the point of view of deformation theory, the symplectic geometry of $X$ can be approached by $H^q(X,\Omega_X^p)$ while the complex geometry of $Y$ can be approached by $H^q(Y,\sT_Y^p)$. If $X$ is ``physically mirror" to $Y$, then one should have $H^q(X,\Omega_X^p)=H^q(Y,\sT_Y^p).$ 
	\begin{definition}
		We say that a pair of Calabi-Yau varieties $(X,Y)$ satisfies the \bf{topological mirror symmetry} if $H^q(X,\Omega_X^p)=H^q(Y,\sT_Y^p)$ for all $p,q$. Moreover, if $Y$ is hyperk\"ahler, the holomorphic symplectic form indeces an isomorphism $\sT_Y\cong\Omega^1_Y$. Thus the above topological mirror test becomes $h^{p,q}(X)=h^{p,q}(Y)$.
	\end{definition}

	Hausel and Thaddeus in \cite{HT03} observe that the Hitchin systems with Langlands dual groups may give the examples of topological mirror symmetry pairs. In \cite{HT01,HT03}, they conjectured that the moduli spaces of $\SL_r$/$\PGL_r$-Higgs bundles are \textbf{topological mirror partners} and proved this for $r=2,3$. The non-parabolic version was proved by Groechenig-Wyss-Ziegler \cite{GWZ20m} for arbitrary rank $r$ and degrees coprime with $r$. They propose a $p$-adic integration formalism to show the equality between twisted stringy Hodge numbers. 
	
	
	In this paper, we prove the topological mirror symmetry for parabolic $\SL_r/\PGL_r$ Hitchin systems with a general parabolic structure (not only full flag cases). For the convenience of statement, we denote moduli of stable parabolic $\SL_r$ Higgs bundles with determinant $\sL$ by $\sM_{\SL_r,P}^{\sL}$ and denote the moduli stack of stable $\PGL_r$ Higgs bundles of degree $e$ by $\sM_{\PGL_r,P}^{e}$. The main result of this paper can be roughly stated as follows (See Theorem \ref{main theorem}):
	\begin{theorem}
		Twisted stringy Hodge numbers of  $\sM_{\SL_r,P}^{\sL}$ are equal to those of $\PGL_r$ Higgs bundles $\sM_{\PGL_r,P}^{e}$ with $e\in \BZ/r\BZ$ if there exists a nonzero integer $\lambda$ such that $e\equiv \lambda\deg\sL \ (\emph{mod}\ \Delta_P)$. Here $\Delta_P$ is an integer depending on parabolic structures.
	\end{theorem}
	And it leads to the following theorem (see Theorem \ref{thm:delta_P is one}) which generalizes the rank 2 and 3 parabolic cases in \cite{Gothen19}
	\begin{theorem}\label{thm:full flag theorem rough}
		If $\Delta_P=1$, for any $\sL$ and $e$, twisted stringy Hodge numbers of $\sM_{\SL_r,P}^{\sL}$ ($\sM_{\PGL_r,P}^{e}$) are equal to non-twisted ones.
		
		In particular, $\Delta_P=1$ holds if there exists a marked point $x$ such that the parabolic type at $x$ is a full flag filtration.
	\end{theorem}    
	Our way to prove the topological mirror symmetry for arbitrary parabolic cases follows from the $p$-adic integration formalism founded by Groechenig, Wyss and Ziegler in \cite{GWZ20m}. We need to point out that, parabolic $\SL_r/\PGL_r$ Hitchin systems are not dual abstract Hitchin systems in the sense of \cite[Defnition 6.8,6.9]{GWZ20m} thus it needs some more efforts to run the p-adic integration formalism in \cite{GWZ20m}.
	
	We first build up a parabolic BNR correspondence over a general field, which generalize the parabolic BNR correspondence over algebraically closed field in \cite[Theorem 6]{SWW22}. See Theorem \ref{parabolic BNR non closed}. This gives the arithmetic duality properties of the parabolic Hitchin systems. We resolved singulairties of generic spectral curves via successive blow-ups. As a bonus, we obtain information of rational points on Hitchin fibers, which plays an important in the proof of Theorem \ref{thm:full flag theorem rough}.
	
	As mentioned before, parabolic $\SL_r/\PGL_r$ Hitchin systems are not dual abstract Hitchin systems in the sense of \cite[Definition 6.8,6.9]{GWZ20m}. Because one of the crucial properties enjoyed by abstract dual Hitchin systems is that there are large open subvarieties (in each of the dual Hitchin systems) which are gerbes banded by ``dual" Picard stacks. To be more precise, here ``large" means that the complement is of codimension greater than one.   In the parabolic setting, the nilpotence of Higgs fields at marked points makes all spectral curves singular. We can only resolve those singular spectral curves over an open subvariety of the Hitchin base space which means that complement of such ``good" subvarieties are of codimension one. However, one of the key consequence of the codimension 2 condition in the definition of abstract dual Hitchin systems is the existence of a nowhere vanishing top forms on moduli spaces(stacks). This can be compensated by the existence of symplectic structures (Proposition \ref{symplecticSL}) and then followed by a codimension 2 fixed points arguments as in Lemma \ref{lem:fixed point}. With all these at hands, we show that the $p$-adic integration formalism works well and get our main theorem \ref{main theorem}.

	
	Let us now briefly indicate how this relates to previous works. After the fundamental work of Hausel and Thaddeus \cite{HT01,HT03} mentioned above, several papers such as \cite{Hitchin01,DPA08,Biswas12SYZ,Gothen19,Derryberry20,GWZ20g,LD21,MS21e,MS21o,Hoskins22} investigated various mirror symmetry properties for the Hitchin systems. We refer the reader to those papers and the references therein. The mirror symmetry conjecture of Hausel and Thaddeus (with out parabolic structure) was proved by Groechenig, Wyss and Ziegler in \cite{GWZ20m} using $p$-adic integration and then by  Loeser and Wyss \cite{LW21} using motivic integration. Maulik and Shen  \cite{MS21e,MS21o} give a new proof of the non-parabolic topological mirror symmetry conjecture (and more on the structure of cohomologies)  using perverse sheaves, support theorems for Hitchin fibrations and vanishing cycles. In \cite{Gothen19},  Gothen and Oliveira proved the parabolic topological mirror symmetry conjecture (without twist) for rank $2,3$ with full flag parabolic structures. And they managed to calculate a large part of those stringy Hodge numbers.


	The paper is organized as follows: in Section 2, we first give defintions of parabolic $\SL_r/\PGL_r$ Higgs bundles, and corresponding Hitchin maps. Since we work over an arbitrary field, we use successive blow-ups to resolve singularities and prove a parabolic Beauville-Narasimhan-Ramanan correspondence. As a corollary, we obtain a numerical invariant $\Delta_P$ (as mentioned above) which only depends on parabolic type and affects the arithmetic properties of gerbes. In Section 3, we construct the natural symplectic structure on parabolic $\SL_r$ Higgs bundles, which defines a nowhere vanishing top form used in the calculation of p-adic integrations. In Section 4, we prove the topological mirror symmetry for moduli of parabolic $\SL_r/\PGL_r$ Higgs bundles.

	\noindent \textbf{Acknowledgement}: The authors thank Prof. Yongbin Ruan for his gentle help. The authors also thank Dr. Yaoxiong Wen and Dr. Weiqiang He for helpful discussions. Part of this manuscript was written during the 1st and 3rd author's visit at the Institute for Advanced Study in Mathematics at Zhejiang University. We express our special thanks to the institute for its wonderful environment and support. 
	
	The work of Xiaoyu Su and Xueqing Wen was performed as the Yau Mathematical Sciences Center and supported by Tsinghua Postdoctoral daily Foundation. The work of Bin Wang was performed at the Steklov International Mathematical Center, Moscow, Russia and
	supported by the Ministry of Science and Higher Education of the Russian Federation (agreement no.  075-15-2019-1614 ). 
	
	\section{Parabolic Hitchin Systems}
	In this section, we first introduce the definition of parabolic vector bundles, parabolic Higgs bundles and corresponding parabolic Hitchin maps over an arbitrary field $k$. Then we prove that generic fibers are torsors over certain Abelian varieties via resolution of generic singular curves. Since we use succesive blow-ups construction, all these results hold over an arbitrary field and  we can derive the existence of rational points on certain torsors of dual Prym varieties. The existence of rational points later will play an important role in the arithmetic properties of integrations over a $p$-adic field.

	To apply the formalism in \cite{GWZ20m}, i.e., point counting/$p$-adic integration determines stringy Hodge numbers, a priory one should consider the moduli (spaces/stacks) over a finitely generated $\BZ$-algebra contained in $\BC$. However the existence is sufficient since in the computation of $p$-adic integration along Hitchin fibers, we only need to use the geometry of generic fibers over a ($p$-adic)field. Thus in this section we just need to detect the geometry of the parabolic Hitchin systems over a field.


	\subsection{Parabolic vector bundles and Higgs bundles}\ 
	
	Let $X$ be a smooth projective (geometrically connected) curve of genus $g$ and  over an arbitrary field $k$. Let $\bar k$ be the algebraic closure of $k$ and  $X_{\bar k}$ be the base change $X\times_{k}{\bar k}$. We fix a finite subset $D\subset X(k)$, which we shall also regard as a reduced effective divisor on $X$ and $X_{\bar k}$. We then require that $2g-2+\deg D>0$. We also fix a positive integer $r$ which will be the rank of  vector bundles on $X$ (resp. $X_{\bar k}$).

	
	To define a parabolic structure on a vector bundle we first need to specify a quasi-parabolic structure and weights for each $x\in D$. Quasi-parabolic structure consists of a finite sequence $m^\pt (x)=(m^1(x), m^2(x), \dots, m^{\sigma_x}(x))$ of positive integers summing up to $r$ which we denote this simply by $P$. Weights are given by a choice of a set of real numbers $0\leq \alpha_1(x)<\cdots <\alpha_{\sigma_x}(x)<1$ and we denote weights as $\alpha$. In the following, we will use $(X, D, P, \alpha)$ to denote the parabolic type.
	


	A parabolic vector bundle of type $(X, D, P, \alpha)$ is a rank $r$ vector bundle $\mathcal{E}$ on $X$ which for every  $x\in D$ is endowed with a filtration $\mathcal{E}|_{x}=F^0(x)\supset F^1(x)\supset \cdots \supset F^{\sigma_x}(x)=0$ such that $\dim F^{j-1}(x)/ F^{j}(x)=m^j(x)$ and weights $\alpha$. A parabolic Higgs bundle is a pair $(\mathcal{E}, \theta)$ where $\mathcal{E}$ is a parabolic vector bundle as above and $\theta$ is an $\sO_{X}$-homomorphism  $\theta : \mathcal{E}\to \mathcal{E}\otimes_{\sO_X} \omega_{X}(D)$ with the property that it takes each $F^j(x)$ to $F^{j+1}(x)\otimes_{\sO_{X}}\omega_{X}(D)|_{x}$.

	An endomorphism of the parabolic bundle $\sE$ is a vector bundle endomorphism of $E$ which preserves the filtrations $F^\pt(x)$. We call this a \emph{strongly parabolic endomorphism} if it takes $F^i(x)$ to $F^{i+1}(x)$ for all $x\in D$ and $i$. We denote the subspaces of $\End_{\sO_X}(E)$ defined by these properties as
	$$
	ParEnd(\sE) \text{  resp.\  } SParEnd(\sE).
	$$ 
	Similarly we can define the sheaf of parabolic endomorphisms and sheaf of strongly parabolic endomorphisms, denoted by $\sP ar\sE nd(\sE)$ and  $\sS\sP ar\sE nd(\sE)$ respectively.

	We now define the \emph{parabolic degree (or $\alpha$-degree)} of $\sE$ to be 
	\[\pdeg(\sE):=\deg(E)+\sum_{x\in D}\sum_{j=1}^{\sigma_x}\alpha_{j}(x)m^j(x).\]
	
	\begin{definition}
		A parabolic vector bundle $\sE$ is said to be stable(resp. semistable),  if for every proper coherent $\sO_X$-submodule $F\subsetneq E$ , we have 
		\[
		\frac{\pdeg(\sF)}{r(\sF)}<\frac{\pdeg(\sE)}{r}\ (\text{resp.}\leq),
		\] 
		where the parabolic structure on $\sF$ is  inherited from $\sE$. 
	\end{definition}
	
	With the help of geometric invariant theory, one can construct the moduli space of semistable parabolic Higgs bundles which we denote as $\mathcal{M}_P$(See \cite{Yo93C} for example). When the weights $\alpha$ are chosen generically, the notion of semistable and stable for parabolic Higgs bundle are coincide and in this case, the moduli space $\mathcal{M}_P$ is a smooth quasi-projective variety. In the rest of this paper, we always assume $\alpha$ are generic.

	
	
	
	In this paper, we consider not only $\operatorname{GL_r}$ parabolic Higgs bundles, but also parabolic Higgs bundles with structure groups $\operatorname{SL_r}$ and $\operatorname{PGL_r}$. 
	
	\bdefinition
	Let $(X,D,P,\alpha)$ be a parabolic data. we define an $L$-twisted $\omega_X(D)$-valued $\SL_r$ parabolic Higgs bundle to be a rank $r$ parabolic Higgs bundle $(E,\theta)$ with $\det(E)=L$ and trace free $\theta\in \Gamma(X,\sS\sP ar\sE nd^0(\sE)\otimes\omega_X(D))$, where $L$ is an line bundle on $X$. A $L$-twisted $\omega_X(D)$-valued $\SL_r$ parabolic Higgs bundle is said to be stable(resp. semistable) if it is stable(resp. semitable) as a $\operatorname{GL_n}$ parabolic Higgs bundle. Let $e\in H^2(X_{\bar k},\mu_r)\cong \BZ/r\BZ$, we define an $e$-twisted  parabolic $\PGL_r$ Higgs bundle to be a pair $(\BP(E),\theta)$ with $(E, \theta)$ a rank $r$ parabolic Higgs bundle such that $c_1(E)\equiv e\ (\tx{mod}\ r)$ and $\theta\in \Gamma(X,\sS\sP ar\sE nd^0(\sE)\otimes\omega_X(D))$. We say $(\BP(E),\theta)$ is stable(resp. semistable) if $(E, \theta)$ is stable (resp. semistable).
	\edefinition
	\begin{remark}
		Our definition of $\sM_{\PGL_r,P}^{e}$ implies that the image of an element in $\sM_{\PGL_r,P}^{e}$ of the natural map $H^{1}(X,\PGL_r)\rightarrow H^{2}(X,\mu_r)$ is $-e\mod(r)$.
	\end{remark}
	
	
	For the parabolic data $(X,D,P,\alpha)$, again we assume $\alpha$ is generic, then the moduli space $\mathcal{M}^{\sL}_{\SL_r,P}$ of stable $L$-twisted $\omega_X(D)$-valued $\SL_r$ parabolic Higgs bundle can be seen as a smooth quasiprojective subvariety of $\mathcal{M}_P$ and the moduli stack of stable $\lambda$-twisted $\PGL_r$ parabolic Higgs bundle $\mathcal{M}^{e}_{\PGL_r,P}$ can be realized as the quotient stack $\left[\sM_{\SL_r,P}^{\sL}/\Pic^0(X)[r]\right]$ when $\text{deg}\mathcal{L}\equiv e (\tx{mod}\ r)$.
	


	\subsection{The parabolic Hitchin maps} \label{sec 3}\
	
	Parabolic Hitchin maps are defined by Yokogawa \cite[Page 495]{Yo93C} (see also \cite{BK18} and \cite[Section 3]{SWW22}). It is defined as a restriction of characteristic polynomial map.
	Firstly we can define a projective morphism (we call it as the characteristic polynomial map) from the moduli space (and even moduli stack) of Higgs bundles to an affine space  
	$\prod_{i=1}^r \mathbf{H}^0(X,(\omega_X(D))^{\otimes i}))$ and point wisely given by the characteristic polynomial of the parabolic Higgs field ${\theta}$ as
	\[
	\tx{char}_P:\sM_P\to \mathbf{H}_D:= \prod_{i=1}^r \mathbf{H}^0(X,(\omega_X(D))^{\otimes i}),\ \ (\sE,\theta)\mapsto (a_1(\theta),\cdots, a_n(\theta)),
	\]
	where $a_i$'s are the coefficients of the characteristic polynomials. 
	By calculation in \cite{BK18} (see also \cite[Subection 3.1]{SWW22}), the image of $
	\tx{char}_P$ lies in a subspace $\mathbf{H}_P$ determined by the parabolic type $P$. The precise definition is as follows (see also \cite[Subection 3.1]{SWW22} and \cite{BK18}).
	
	For each $x\in D$, we rearrange $(m^1(x), \cdots, m^{\sigma_x}(x))$ to be $n_1(x)\geq n_2(x)\geq \cdots \geq n_{\sigma_x}$. So $\{n_i(x)\}$ is a partition of $r$ and we use $\mu_j(x)$ to denote the dual partition of it, which mean that $\mu_{j}(x)=\#\{\ell:n_{\ell}(x)\geq j, 1\leq\ell\leq \sigma_x\}$. Now we assign a level function $j\rightarrow \gamma_{j}(x), 1\leq j\leq r$, such that $\gamma_{j}=l$ if and only if $$\sum_{t\leq l-1}\mu_{t}(x)< j\leq\sum_{t\leq l}\mu_{t}(x).$$
	
	\begin{def-prop}[\cite{SWW22}, Theorem 4]
		The image of $\tx{char}_P:\mathcal{M}_P\rightarrow \mathbf{H}_D$ lies in the following subspace:
		\[
		\sH_{P}:=\prod_{j=1}^r\mathbf{H}^{0}\Big(X,\omega_X^{\otimes j}\otimes\sO_X\big(\sum_{x\in D}(j-\gamma_{j}(x))\cdot x\big) \Big)\subset \mathbf{H}_D
		\] 
		which we will call it the $\GL_r$-\emph{parabolic Hitchin base}, and we will call $h_P: \mathcal{M}_P\rightarrow \sH_P$ the $\GL_r$-\emph{parabolic Hitchin map}.
	\end{def-prop}
	
	
	
	For $\operatorname{SL}_r$ case, the moduli space of parabolic $\SL_r$ Higgs bundles $\mathcal{M}^{\sL}_{\SL_r,P}$ is a subspace in $\sM_P$, and the Higgs fields are all trace free. So we define the $\SL_r$-parabolic Hitchin base as $\sH_P^{0}:=\prod_{i=2}^r \mathbf H^{0}\Big(X,\omega_X^{\otimes j}\otimes\sO_X\big(\sum_{x\in D}(j-\gamma_{j}(x))\cdot x\big) \Big)$ the trace free part of the $\GL_r$-parabolic Hitchin base. We then define the $\SL_r$-parabolic Hitchin map as the restriction 
	\[h_{\SL_r,P}:\sM^{\sL}_{\SL_r,P}\to \sH^0_P,\ (\sE,\theta)\mapsto (a_2(\theta),\cdots, a_n(\theta))\]
	of $h_P$. From the definition, we see that the $\SL_r$-parabolic Hitchin map is invariant under the twisting action given by tensoring an $r$-torsion line bundle, then it follows that the $\SL_r$-parabolic Hitchin map factor through the quotient stack $\left[\sM_{\SL_r,P}^{\sL}/\Pic^0(X)[r]\right]=\sM_{\PGL_r,P}^{\deg\sL}$ and we define the $\PGL_r$-parabolic Hitchin map as the quotient of $\SL_r$ ones. Then the Hithcin bases of $\operatorname{SL}_r$ and $\operatorname{PGL}_r$
	parabolic Higgs bundles are both $\sH_P^{0}$ and we have the following picture:
	\[\begin{tikzcd}
		\mathcal{M}^{\sL}_{\SL_r,P}\arrow[rr,"\pt/\Gamma"]\arrow[rd,"h_{\SL_r,P}"'] &  & \mathcal{M}^{e}_{\PGL_r,P} \arrow[dl,"h_{\PGL_r,P}"] \\
		&	\sH_P^{0} & 
	\end{tikzcd}
	\]
	In what follows, if the parabolic data are fixed, we will denote the moduli space of parabolic $\SL_r$-Higgs bundles as $\hat{\sM}$ and the corresponding Hitchin map as $\hat h$ and denote the moduli stack of parabolic $\PGL_r$-Higgs bundles as $\check{\sM}$ and the corresponding Hitchin map as $\check h$. If we would stress the data of the moduli spaces (or stacks) we will use the complete version such as $\mathcal{M}^{\sL}_{\SL_r,P}$ and $\mathcal{M}^{e}_{\PGL_r,P}$.
	


	\subsection{Generic fibers of the parabolic Hitchin map}

	To detect the generic fiber of a (parabolic) Hitchin map, one of the most efficient method is the so called ``BNR correspondence" established by Beauville, Narasimhan and Ramanan in \cite{BNR}. 
The classical BNR correspondence is a one-to-one correspondence between Higgs bundles (with out parabolic structure) with a fixed characteristic polynomial and torsion-free rank $1$ sheaves on the spectral curve (with defining equation the fixed characteristic polynomial given by a closed point in the Hitchin base) when the spectral curve is integral. In particular, if the spectral curve is smooth, the Hitchin fiber can be identified with the Picard variety over the corresponding spectral curve via the BNR correspondence.

However, things are more complicated in the parabolic case. The spectral curves (c.f. \cite[Subsection 3.2 and Subsection 4.2]{SWW22}) would never be smooth unless all the parabolic structure are Borel. In the case that the generic spectral curve is singular, there is no obvious way to endow rank $1$ torsion-free sheaves on the (singular) spectral curve a prescribed parabolic structure fits well with the Higgs fields. In \cite[Theorem 6]{SWW22}, via some not-obvious commutative algebra arguments, we set up a BNR type correspondence in the parabolic case:


\begin{proposition}[parabolic BNR correspondence, Theorem 6 in \cite{SWW22}]\label{parabolic BNR}
	Assume that the base field $k$ is algebraically closed. For generic choice $a\in \sH_P$, if $X_a$ is integral, then there is a one-to-one correspondence between:
	\begin{itemize}
		\item[(1)] Parabolic Higgs bundles in the fiber $h_P^{-1}(a)$;
		\item[(2)] Line bundles on the normalization $\tilde{X}_a$ of the spectral curve $X_a$.
	\end{itemize}
\end{proposition}

\begin{remark}
	In \cite{SWW22}, we assume that the genus $g\geq 2$, however, since our computations are local, so it is valid in the cases $g=0$ and $g=1$. But in the case $g=0$, there would be parabolic Hitchin bases so that all spectral curves are not integral, so we add the assumption that the spectral curve should be integral in our statement in Proposition \ref{parabolic BNR} than  Theorem 6 in \cite{SWW22}.
\end{remark}

\subsubsection{Resolve Singularities of Spectral curves}

As mentioned in the introduction, to apply the $p$-adic integration method, we need to set up a parabolic BNR correspondece over arbitrary fields. The first step is to resolve the singular spectral curves by successive blow-ups. It is more delicate than the toric resolution used in \cite{SWW22} which seems only work well over algebraically closed fields. As a bonus of successive blow-ups, Corollary \ref{cor:k-rational point} about rational points is a new phenomena for non algebrically closed fields.


\begin{definition}(\footnote{
		For more details, we refer the readers to \cite[Section 3]{BNR}. })
	For any $a\in \mathcal{H}_P$, we regard it as a characteristic polynomial. Then the spectral curve $X_a$ defined by $a$ is the zero locus of that polynomial in $\mathbb P_X(\sO_X\oplus \omega_X(D))$. The spectral curve $X_a$ is finite and flat over the base curve $X$ and we denote the projection by $\pi_a: X_a\rightarrow X$.
\end{definition}

First of all, we assume that there exists $a\in \sH_P$ such that $X_a$ is integral\footnote{It is possible that if $g(X)=0$, for certain parabolic types, all the spectral curves are reducible.}. Though all spectral curves are singular, there still exists an open subset $U$ such that for all $a\in U\subset \sH_{P}$, $X_{a}$ is integral, totally ramified at $x\in D$ and smooth elsewhere, as shown in the Appendix in \cite{SWW22}. In the following, we only consider $a\in U$.

Similar as in \cite[Subsection 4.2]{SWW22}, we denote the normalization of $X_a$ by $\tilde{X}_a$ and natural map $\tilde{X}_a\rightarrow X$ by $\tilde{\pi}_a$.  To analyse $\tilde{X}_a$, we focus at an $x\in D$, and use $\sO$ to denote the formal local ring at $x$ and choose a local coordinate $t$ in a formal neighborhood of $x$ and then use $\frac{dt}{t}$ to trivilize $\omega_X(D)$. Now we can cover $X_a$ by $X_a-\{\pi^{-1}(x)\}$ and $\Spec A$, where $A=\sO[\lambda]/(f)$ is the formal local completion of the local ring of $X_a$ at $\pi^{-1}(x)$. Then we only need to understand the normalization of $A$.

Since we focus at $x\in D$, for convenience, we may write $\gamma_i(x)$, $\mu_i(x)$, $\sigma_x$ as $\gamma_i$, $\mu_i$, $\sigma$. Then  $a_i=c_it^{\gamma_i}\in k[[t]]$, with $c_i(0)\neq 0$ and the coordinate ring of $X_a$ around $x$ becomes: 
\[
A^{(0)}:=A=\frac{k[[t]][\lambda]}{(\lambda^{r}+c_{1}t^{\gamma_1}\lambda^{r-1}+\cdots+c_{r-1}t^{\gamma_{r-1}}\lambda +c_{r}t^{\gamma_r})}
\]

We begin with a combinatorial lemma. For $0\leq i \leq \sigma-1$, we define $N_i=\sum_{j=i+1}^{\sigma}n_j$ and $N_{\sigma}=0$.

\begin{lemma}\label{successive degree}
	
	For $1\leq i \leq \sigma$, we have $$\min_{1\leq \ell \leq r}\{i\gamma_\ell+N_{i-1}-\ell\}=n_{i}$$ Moreover, we denote the level set $\{\ell\mid i\gamma_{\ell}+N_{i-1}-\ell=n_{i}\}$ by $L_{i}$, then
	\begin{itemize}
		\item[(a)] $\exists \ell\in L_{i}$, $i\gamma_{\ell}+N_{i-1}-\ell=n_{i}$ and $(i-1)\gamma_{\ell}+N_{i-1}-\ell=0$.
		\item[(b)] $\max_{\ell\in L_{i}}\{\gamma_\ell\}-\min_{\ell\in L_{i}}\{\gamma_{\ell}\}=\# \{\ell\mid \mu_{\ell}=i\}$.
	\end{itemize}
\end{lemma}
\begin{proof}
	Let's fix $i$, for $1\leq \ell\leq r$, we define:
	$$d_{\ell}=i\gamma_\ell+N_{i-1}-\ell.$$
	To determine the minimal value of $d_{\ell}$, we first compare $d_{\ell}$ and $d_{\ell+1}$,
	\[d_{\ell}-d_{\ell+1}=i(\gamma_{\ell}-\gamma_{\ell+1})+1\]
	Since $i\geq 1$ and $-1\leq \gamma_{\ell}-\gamma_{\ell+1}\leq 0$, one has $d_{\ell}-d_{\ell+1}=1$ or $d_{\ell}-d_{\ell+1}=-i+1$. By the intermezzo, $\gamma_{\ell}-\gamma_{\ell+1}=-1$ if and only if $\ell=\sum_{j=1}^{q} \mu_{j}$ for some $q$. 
	Then $d_{l}$ attains minimal value when $\ell=\sum_{\{\mu_j\geq i\}}\mu_j$, then $\gamma_{\ell}=n_{i}$ and $d_\ell=i\cdot n_i+N_{i-1}-\ell$. Since $n_{i}=\#\{j\mid \mu_{j}\geq i\}$, 
	\[
	N_{i-1}+i\cdot n_{i}-\sum_{\{\mu_j\geq i\}}\mu_j=N_{i-1}-\sum_{\{\mu_j\geq i\}}(\mu_{j}-i)=N_{i-1}-\sum_{j=i+1}^{\sigma}n_{j}=n_{i}
	\]
	
	Let $\alpha=\max L_{i}, \beta=\min L_{i}$, it is not difficult to see $\gamma_{\alpha}-\gamma_{\beta}=\# \{\ell\mid \mu_{\ell}=i\}$.    	
\end{proof}
We now define $$A^{(i)}:=\frac{k[[t]][\lambda,u_i]}{(t-\lambda^i u_i,\lambda^{N_i}+\cdots +c_{r-1}u_i^{\gamma_{r-1}}\lambda^{i\gamma_{r-1}+N_i-r+1}+c_ru_i^{\gamma_r}\lambda^{i\gamma_r+N_i-r})}$$ where $0\leq i\leq \sigma$, $N_i=\sum_{j=i+1}^{\sigma}n_j$, $N_{\sigma}=0$.
Then we have:
\begin{proposition}\label{ring description}
	
	\begin{itemize}
		\item[(a)] The singular point of $\Spec(A^{(i)})$ is $(\lambda, u_i)=(0,0)$, in particular, $\Spec A^{(\sigma)}$ is smooth.
		\item[(b)] the ramifications of closed points (need not to be $k$-rational) on $\tilde{X}_a$ over $x$ are $\{\mu_1,\mu_2,\ldots,\mu_{\sigma}\}$ which as before is the conjugate partition.
		\item[(c)] If there is one $\mu_i$ different form others, i.e. $\# \{\ell\mid \mu_{\ell}=i\}=1$, we have a $k$-rational point on $\tilde{X}_a$. More generally, for each $i$, there is a line bundle of degree $\# \{\ell\mid \mu_{\ell}=i\}$ on $\tilde{X}_a$ which is defined over $k$.
	\end{itemize}
	
\end{proposition}

\begin{proof}
	We proceed by induction on $i$, the case $i=0$ is obvious. 
	
	Assume the proposition is true for $i-1$, then we blow up $A^{(i-1)}$ at the ideal $(\lambda,u_{i-1})$. We need to solve two equations: 
	
	\begin{equation}\label{blowupi.1}\left\{\begin{array}{l}
			\lambda^{N_{i-1}}+\sum_{\ell=1}^{r} c_\ell u_{i-1}^{\gamma_\ell}\lambda^{(i-1)\gamma_\ell+N_{i-1}-\ell}=0 \\
			\lambda=u_{i-1}v_{i}
		\end{array}
		\right.
	\end{equation}
	
	\begin{equation}\label{blowupi.2}
		\left\{\begin{array}{l}
			\lambda^{N_{i-1}}+\sum_{\ell=1}^{r} c_\ell u_{i-1}^{\gamma_\ell}\lambda^{(i-1)\gamma_\ell+N_{i-1}-\ell}=0 \\
			u_{i-1}=\lambda u_{i}
		\end{array}
		\right.
	\end{equation}
	
	Substitute $\lambda=u_{i-1}v_{i}$, the first equation becomes    	
	$$u_{i-1}^{N_{i-1}}v_{i}^{N_{i-1}}+\sum_{\ell=1}^{r}c_\ell\cdot  u_{i-1}^{i\gamma_\ell+N_{i-1}-\ell}\cdot v_{i}^{(i-1)\gamma_\ell+N_{i-1}-\ell}=0.$$
	By Lemma \ref{successive degree}, the exceptional divisor is $u_{i-1}^{n_{i}}=0$. Eliminate it, by Lemma \ref{successive degree} (a), there exists a non-zero constant term. Then,  $u_{i-1}=0, v_{i}=0$ is not a solution of this equation.    	
	
	So we only need to analyse the second piece. By Lemma \ref{successive degree}, the exceptional divisor is $\lambda^{n_{i}}=0$. The equation of strict transform becomes: $$\lambda^{N_{i}}+c_1u_{i}^{\gamma_1}\lambda^{i\gamma_1+N_{i}-1}+\cdots +c_ru_{i}^{\gamma_r}\lambda^{i\gamma_r+N_{i}-r}=0$$ Combine these two equations together, we see that the coordinate ring of $X_a^{(i)}$ is of the form presented in the proposition. 
	
	Remain to analyse the singular locus,  In the equation above, when $t=0$ then $\lambda=0$, we get: 
	\begin{equation}\label{eq:ramification point}
		\sum_{\ell\in L_{i}} c_{\ell}(0)\cdot u_{i}^{\gamma_{\ell}}=0
	\end{equation}
	the generality of $\{c_i\}$ would ensure  nonzero roots of \eqref{eq:ramification point} are simple. 
	
	By Jacobian criterion, the only singular point of $A^{(i)}$ is $\lambda=0, u_{i}=0$. When $i=\sigma$, $N_{\sigma}=0$ so there is no singular point in $A^{(\sigma)}$.
	
	We can see that the ramification index of smooth points on $\Spec A^{(i)}$ over $x$ is $i$. In particular, 
	By Lemma \ref{successive degree} (b), if $\#\{\ell\mid \mu_{\ell}=i\}\ne 0$, then we have points of ramification index $i$ over $x$, and [(b)] follows. Since $k$ is not algebraically closed, the nonzero solution of \eqref{eq:ramification point} may not lie in $k$. But the equation \eqref{eq:ramification point} defines a divisor of degree $\#\{\ell\mid \mu_{\ell}=i\}$ on $\Spec A^{(i)}$ hence a line bundle of same degree on $\tilde{X}_a$. Also, if there is one $\mu_i$ different from others, the equation \eqref{eq:ramification point} reduces to a linear equation by Lemma \ref{successive degree} (b), hence we have a $k$-rational point. [(c)] follows.
\end{proof}

\begin{remark}
	Successive blow-ups are used to show that the ramification indexes of points on $\tilde{X}_a$ over $x\in D$ are exactly $\{\mu_1(x),\ldots,\mu_{\sigma_x}(x)\}$. Even though, these points may not defined over $k$, hence we don't know the exact number of $k$-points over $x\in D$. But this makes it possible for us to define a natural filtration on the direct image of line bundles of $\tilde{X}_a$.
\end{remark}

We write $\sO_{\tilde{X}_a,x}$ as the completion of $\sO_{\tilde{X}_a}$ at the points over $x\in X$.
\begin{proposition}\label{propdef:young filtration}
	There is a canonical filtration $F$ of $\sO_{X,x}$-modules on $\sO_{\bar{X}_a,x}$ such that :
	\[
	\dim \frac{F^{i-1}\sO_{\tilde{X}_a,x}}{F^i\sO_{\tilde{X}_a,x}}=m_i(x)
	\]
\end{proposition}
\begin{proof}
	Since $\tilde{X}_a$ is smooth, we may rewrite $\sO_{\tilde{X}_a,x}$ as a direct sum of the following form:
	\[
	\sO_{\tilde{X}_a,x}=\oplus R_i
	\]
	where each $R_i$ is the direct sum of local rings (also DVR) with ramification index $i$ which is a semi-local ring. Notice that for simplicity we put $R_i=0$ if $\#\{\ell|\mu_{\ell}=i\}=0$. 
	
	If we write $R_i=\oplus R_{ij}$ where each $R_{ij}$ is the local ring at closed points over $x$, then we can define the filtration on $R_i$ as:
	\[
	R_i=\oplus R_{ij}\supset \oplus\mathfrak{m}_{ij}R_{ij}\supset\cdots\supset \oplus \mathfrak{m}_{ij}^{i}R_{ij}=\mathfrak{m}_xR_i
	\]
	By Lemma \ref{successive degree}(b), and the equation \eqref{eq:ramification point}, we know that:
	\[
	\dim \frac{\oplus\mathfrak{m}^{\ell}_{ij}R_{ij}}{\oplus\mathfrak{m}^{\ell}_{ij}R_{ij}}=\#\{\ell|\mu_{\ell}=i\}
	\]
	Now the proposition follows from the following equality:
	\[
	m_i(x)=\sum_{k\ge i}\#\{\ell|\mu_{\ell}=k\}.
	\]\end{proof}
The last thing about the geometry of generic fibers is concerning the existence of rational points which has an interesting effect on the arithmetic property of gerbes.
\begin{definition}
	We denote $\Delta_{P}:=\gcd\big \{\#\{\ell|\mu_{\ell}(x)=i\}\big\}_{i=1,\ldots,r;\  x\in D}$.
\end{definition}
Then from Proposition \ref{ring description}[(c)], we can obtain:
\begin{corollary}\label{cor:k-rational point}
	There is a $k$ rational point on $\Pic^{\Delta_P}(\tilde{X}_a)$.
\end{corollary}
\begin{proof}
	By Proposition \ref{ring description}[(c)], we know that there is a $k$ rational point in $\Pic^{\#\{\ell|\mu_{\ell}(x)=i\}}(\bar{X}_a)$ for each $x$ and $i$. Hence we prove the corollary.
\end{proof}

\subsubsection{Geometry of the Parabolic Hitchin System}

Based on the parabolic BNR correspondence in \cite[Theorem 6]{SWW22}, and Proposition \ref{propdef:young filtration}, we now can generalize our parabolic BNR correspondence to an arbitrary field.
\begin{theorem}\label{parabolic BNR non closed}
	Let $k$ be a field and $X$ a smooth geometric integral curve over $k$ with parabolic structure Let $(X,D,P,\alpha)$. There exists an open subset $\sH_{P}^{\diamond}$ in the $\GL_r$-parabolic Hitchin base $\sH_{P}$ such that when restricts to $\sH_{P}^{\diamond}$, $h_{P}$ is a torsor over the relative Jacobian $\Jac(\tilde{\sX}/\sH_{P}^{\diamond})$ where $\tilde{\sX}$ is the normalization of the universal spectral curve $\sX$ over $\sH_{P}^{\diamond}$.
\end{theorem}
\begin{proof}
	We base change to the function field of parabolic Hitchin base and then apply Proposition \ref{ring description} to get the normalization of the universal spectral curve. Under the open condition that all the nonzero roots of \eqref{eq:ramification point} are simple, we can find an open subset $\sU'$ such that the following functor is well defined:
	$$\begin{array}{rccc}\tilde{\pi}_*:&\{\text{degree } \delta \text{ line bundles over } \tilde{\sX}\}&\to& \left\{\begin{array}{c}\text{Parabolic Higgs bundle } (\mathcal{E},\theta)\\
			deg(E)=d,  	\text{ over} X\times \sH_{P}^{\diamond}
		\end{array} \right\}\\
		&&&\\
		& L&\mapsto &\begin{array}{c}
			\tilde\pi_*L
		\end{array}
	\end{array}$$
	where $\delta=(r^2-r)(g-1)+\sum_{x\in D}\dim(G/P_x)+d$ and parabolic structure on $\tilde\pi_*L$ is given by the Young diagram filtration as in Proposition \ref{propdef:young filtration} and parabolic Higgs field is given by the $\sO_{\tilde{\sX}}$-module structure.
	
	As we have shown in \cite[Theorem 6]{SWW22}, this is an isomorphism after base change to $\bar{k}$. Hence, it is also an isomorphism over $k$. 
\end{proof}
\begin{corollary}\label{cor:deg 1 div}
	If $\Delta_{P}=1$, then the $k$-rational point of $h_P^{-1}(a)$ is nonempty for any $a\in\sH_{P}^{\diamond}(k)$
\end{corollary}
\begin{proof}
	By Corollary \ref{cor:k-rational point}, there is a section of $h_P$ over $\sU$. This implies that the degree map $\deg:\Pic(\tilde{X}_a)\rightarrow\BZ$ is surjective. Combined with the above theorem, we get the result. In fact, this also implies that the torsor is trivial.
\end{proof}
We now give a description of generic fibers for the moduli of $\operatorname{SL_r}$ and $\operatorname{PGL_r}$ parabolic Higgs bundles via Prym varieties. First recall that for the finite covering, $\tilde{X}_a\xrightarrow{\tilde{\pi}}X$, the norm map is defined as:
\begin{align*}
	\Nm:\Pic(\tilde{X}_a)&\xrightarrow{\tilde{\pi}} \Pic(X)\\
	\sW&\rightarrow\det(\tilde{\pi}_*\sW)\otimes\det(\tilde{\pi}_*\sO_{\tilde{X}_a})^{-1}.
\end{align*}
We define the Prym variety as:
\[
\Prym(\tilde{X}_a/X):=\ker\Nm\subset\Pic^0(\tilde{X}_a)
\]
and it is known that $\Prym(\tilde{X}_a/X)$ and $\Prym(\tilde{X}_a/X)/\Pic^0(X)[r]$ are dual abelian varieties. For simplicity, we put $\Prym^{\vee}(\tilde{X}_a/X):=\Prym(\tilde{X}_a/X)/\Pic^0(X)[r]$

Since for $\sW,\sW'$ of same degree over $X$, $\Nm^{-1}(\sW)\cong\Nm^{-1}(\sW')$, we may simply denoted them by $\Prym^{d}(\tilde{X}_a/X)$ where $d=\deg \sW=\deg\sW'$. As a Corollary of Theorem \ref{parabolic BNR non closed}, and argue as \cite[Lemma 7.8]{GWZ20m}:
\begin{proposition}\label{prop:fiber of sl and pgl}
	we denote the trace free part of $\sH_{P}^{\diamond}$ by $\sH_{P}^{0\diamond}=\sH_{P}^{\diamond}\cap \sH_P^0$. 
	\begin{enumerate}
		\item For the $\operatorname{SL_r}$ case, the Hitchin map $h_{\SL_r,P}$ over $\sH_{P}^{0\diamond}$ is isomorphic to $\Prym^{\deg\sL}(\tilde{\sX}/\sH_{P}^{0\diamond})$ which is a torsor over the relative prym variety $\Prym(\tilde{\sX}/\sH_{P}^{0\diamond})$, and similarly
		\item for the $\operatorname{PGL_r}$ case, $h_{\PGL_r,P}$ over $\sH_{P}^{0\diamond}$ is isomorphic to $\Prym^{e}(\tilde{\sX}/\sH_{P}^{0\diamond})/\Pic^0(X)[r]$ which is a torsor over the dual Abelian variety $\Prym^{\vee}(\tilde{\sX}/ \sH_{P}^{0\diamond})$.
	\end{enumerate} 
\end{proposition}\label{prop:dual av fiber}
\begin{proof}
	Since an element in a generic fiber $h_{\SL_r,P}^{-1}(a)$ corresponds to line bundle $\sW$ on $\tilde{X}_a$ which satisfies:
	\[
	\det\tilde{\pi}_*\sW\cong\sL.
	\]
	Hence by the definition of $\Nm$, a generic fiber $h_{\SL_r,P}^{-1}(a)$ is naturally isomorphic, as a $\Prym(\tilde{\sX}/\sH_{P}^{0\diamond})$ torsor, to:
	\[\Nm^{-1}(\sL\otimes\det(\tilde{\pi}_*\sO_{\tilde{X}_a})).\]
	And similar for generic fibers of $h_{\PGL_r,P}$.
\end{proof}

\subsection{Arithmetic duality of the generic Hitchin fibers}\label{subsec:arithmetic duality}

The moduli space of the parabolic $\PGL_r$-Higgs bundles is non-compact and singular, thus to put in into the toppological mirror symmetry picture, the  ``Hodge numbers” must be interpreted in a generalized sense: as stringy mixed Hodge numbers twisted by a gerbe (see \cite[Section 4]{HT03}). In this subsection, we introduce certain naturally defined $\mu_r$-gerbes. And  as in \cite[Theorem 7.18]{GWZ20m}, the arithmetic duality of Hitchin fibers holds which is a refined version of Proposition \ref{prop:fiber of sl and pgl}.



Following the concepts in \cite{HT03} and in particular \cite[Subsection 7.4]{GWZ20m}, we define the $\mu_r$ gerbe $\alpha^{\sL}_{\SL_r,P}$ as the lifting gerbe of the (pullback of the) universal parabolic $\PGL_r$ Higgs bundle over the moduli space $\sM_{\SL_r,P}^{\sL}$. 
Similarly, using the lifting gerbe of unverisal parabolic $\PGL_r$ Higgs bundles over the moduli stack $\sM_{\PGL_r,P}^{e}$, we define a $\mu_r$-gerbe $\alpha_{\PGL_r,P}^{e}$ on $\sM_{\PGL_r,P}^{e}$. In fact, $\alpha_{\PGL_r,P}^{e}$ can also be defined as a descent of $\alpha_{\SL_r,P}^{\sL'}$ on $\sM_{\SL_r,P}^{\sL'}$ for a line bundle $\sL'$ on $X$ of degree $e$. See \cite[Subsection 7.4]{GWZ20m}.

Since we identify generic fibers of $h_{\SL_r,P}$ and $h_{\PGL_r,P}$ with torsors over dual Prym varieties, hence over some finite extension of $k$, $\alpha_{\SL_r,P}^{\sL},\alpha_{\PGL_r,P}^{e}$ split. Such kind of gerbes are defined as arithmetic gerbes in \cite[Subsection 6.1]{GWZ20m}. And the relative splitting of $\alpha_{\SL_r,P}^{\sL}$ with respect to $h_{\SL_r,P}$ over $\sH_{P}^{0\diamond}$ defines a $\Prym^{\vee}(\tilde{\sX}/X)[r]$-torsor denoted by $\Split(\sM_{\SL_r,P}^{\sL},\alpha_{\SL_r,P}^{\sL})$. And naturally, the $\Prym^{\vee}(\tilde{\sX}/X)[r]$-torsor $\Split(\sM_{\SL_r,P}^{\sL},\alpha_{\SL_r,P}^{\sL})$ induces a $\Prym^{\vee}(\tilde{\sX}/X)$-torsor. We may denote it by $\Split'(\sM_{\SL_r,P}^{\sL},\alpha_{\SL_r,P}^{\sL})$. And same arguments for $\PGL_r$-side. Now we close this section by the following theorem:
\begin{theorem}\label{thm:syz mirror symmetry}
	we have the following identification:
	\begin{align*}
		\Split'(\sM^{\sL,\diamond}_{\SL_r,P}/\sH_{P}^{0\diamond}, e\alpha_{\SL_r,P}^{\sL})\cong \sM^{e,\diamond}_{\PGL_r,P}/\sH_{P}^{0\diamond}\\
		\Split'(\sM^{e,\diamond}_{\PGL_r,P}/\sH_{P}^{0\diamond}, d\alpha_{\PGL_r,P}^{e})\cong \sM^{\sL,\diamond}_{\SL_r,P}/\sH_{P}^{0\diamond}
	\end{align*}
	where $\sM^{\sL,\diamond}_{\SL_r,P}$ (resp. $\sM^{e,\diamond}_{\PGL_r,P}$) is the restriction of $\sM^{\sL}_{\SL_r,P}$ (resp. $\sM^{e}_{\PGL_r,P}$) to $\sH_{P}^{0\diamond}$.
\end{theorem}
\begin{remark}
	In our parabolic setting, the universal parabolic $\PGL_r$ Higgs bundle over the moduli space $\sM_{\SL_r,P}^{\sL}$ always comes a universal Higgs bundle. Hence $\alpha_{\SL_r,P}^{\sL}$ is a trivial gerbe. 
\end{remark}

\section{Symplectic Structures of $\sM_{\SL_r,P}^{\sL}$}

To prove toplogical mirror symmetry for non-parabolic $\SL_r/\PGL_r$ Higgs moduli spaces over $\BC$ via the p-adic integration formalism in \cite{GWZ20m}, the first step is to descend these varieties from $\BC$ to a finitely generated $\BZ$-algebra contained in $\BC$. Let $R$ be such an algebra. By shrinking $\Spec(R)$, we will assume $R$ is regular. In this section, we consider moduli spaces of parabolic $\SL_r$ Higgs bundles defined over $R$.  For simplicity, we fix our parabolic data $(X,D,P,\alpha)$ and put $\hat{\sM}:=\sM_{\SL_r,P}^{\sL}$,
$\hat h=h_{\SL_r,P}:\hat {\sM}\to \sH_P^0$. we show that there is a symplectic structure on the moduli spaces $\hat{\sM}$. 

Let $\hat{\sN}$ be the moduli space of stable parabolic $\mathcal{L}$-twisted $\SL_r$ bundle on $(X,D,P,\alpha)$.
Then the cotangent bundle $T^\vee_{\hat{\sN}}$ of $\hat{\sN}$ is an open subset $\hat{\sM}$. The idea is to first construct a nowhere vanishing 2-form on $\hat{\sM}$ via deformation theory. Then we show that it restricts to the canonical symplectic form on $T^\vee_{\hat{\sN}}$ which means that it is also a closed 2-form, hence a symplectic form on $\hat{\sM}$.

Let $\bsf V$ be the universal parabolic $\SL_r$ vector bundle \footnote{The existence of $\bsf V$ is due to the generic choice of weight $\alpha$, see \cite[Proposition 3.2]{BY99}} over $X\times\hat{\sN}$ and $(\bsf U,\uptheta_\bsf U)$ be the universal parabolic $\SL_r$ Higgs bundle on $X\times \hat{\sM}$. 
\begin{equation}
	\begin{tikzcd}
		&\bsf V\ar[d]\\
		&X\times \hat{\sN}\ar[ld]\ar[rd,"p_{\hat{\sN}}"]\\
		X&&\hat{\sN}
	\end{tikzcd},
	\begin{tikzcd}
		&(\bsf U,\uptheta_{\bsf U})\ar[d]\\
		&X\times \hat{\sM}\ar[ld,"p_X"]\ar[rd,"p_{\hat{\sM}}"]\\
		X&&\hat{\sM}
	\end{tikzcd}
\end{equation}
Over the open subset $X\times T^\vee_{\hat{\sN}}$, one has that $\bsf U|_{\tx{cotangent of }\hat{N}}\cong \pi_{X\times\hat{\sN}}^*\bsf V$ where $\pi_{X\times\hat{\sN}}: X\times T^{\vee}_{\hat{\sN}}\rightarrow X\times\hat{\sN}$. 

By deformation theory, the tangent complex of $\hat{\sN}$ is given by $Rp_{{\hat{\sN}}*}(\sP ar\sE nd^0(\bsf V)[1])$ which is a coherent sheaf and concentrated at degree zero. Here $\sE nd^0$ means the trace free endomorphisms. Moreover, the
tangent complex of $\hat{\sM}$ is given by $$Rp_{{\hat{\sM}}*}(\sP ar\sE nd^0(\bsf U)\xrightarrow{\tx{ad}_{\uptheta_\bsf{U}}} \sS\sP ar\sE nd^0(\bsf U)\otimes p_X^*\omega_X(D)[1])$$ which is also a coherent sheaf and concentrated at degree zero. Let us denote the complex $\sP ar\sE nd^0(\bsf U)\xrightarrow{\tx{ad}_{\uptheta_\bsf{U}}} \sS\sP ar\sE nd^0(\bsf U)\otimes p_X^*\omega_X(D)$ by $\tx{ad}^\pt_{\uptheta_\bsf U}$. Then we see that the tangent sheaf $\sT_{\hat{\sN}}$ is given by $R^1p_{{\hat{\sN}}*}(\sP ar\sE nd^0(\bsf V))$ and the tangent sheaf $\sT_{\hat{\sM}}$ is given by $R^1p_{{\hat{\sM}}*}(\tx{ad}^\pt_{\uptheta_\bsf U})$. By Serre duality, the cotangent sheaf of $\hat{\sN}$ can be given by $p_{{\hat{\sN}}*}(\sS\sP ar\sE nd^0(\bsf V)\otimes p^*\omega_X(D))$, and by Grothendieck duality, the cotangent sheaf of $\hat{\sM}$ is given by $R^{1}p_{{\hat{\sM}}*}({\tx{ad}^\pt_{\uptheta_\bsf U}}^\vee\otimes p^*\omega_X[-1])$. 

Following \cite[page 140, (3.1)]{Yo95}, for parabolic bundle $\sE$, we have 
\begin{equation}\label{(3.4)}  \sP ar \sE nd(\sE)^{\vee}\cong \sS \sP ar\sE nd(\sE)\otimes_{\sO_X}\sO_X(D).
\end{equation}
In particular, for the universal Higgs bundle, one has the isomorphism $$\phi: \sS\sP ar\sE nd^0(\bsf U)\otimes p_X^*\omega_X(D) \stackrel{\cong}{\to} {\sP ar\sE nd^0(\bsf U)}^\vee\otimes p_X^*\omega_X$$ and it fits into the following commutative diagram:
\begin{footnotesize}
	\[\tilde B:\begin{tikzcd}
		\sP ar\sE nd^0(\bsf U)\arrow[rr,"{\tx{ad}_{\uptheta_\bsf{U}}}"]\arrow[d,"\cong","{-(\phi^\vee)^{-1}}"']&	
		& \sS\sP ar\sE nd^0(\bsf U)\otimes p_X^*\omega_X(D) \arrow[d,"\cong","\phi"'] \\
		(\sS\sP ar\sE nd^0(\bsf U)\otimes p_X^*\omega_X(D))^\vee\otimes p_X^*\omega_X\arrow[rr,"{-\tx{ad}_{\uptheta_\bsf{U}}\otimes\text{id}_{\omega_X}}"]&	
		& {\sP ar\sE nd^0(\bsf U)}^\vee\otimes p_X^*\omega_X
	\end{tikzcd}\]
\end{footnotesize}
where $\tilde B$ means the isomorphism between complexes $\tx{ad}^\pt_{\uptheta_\bsf U}\cong {\tx{ad}^\pt_{\uptheta_\bsf U}}^\vee\otimes p^*\omega_X[-1]$. After taking direct image $R^1p_{{\hat{\sM}}*}$, we easily get the isomorphism $R^1p_{{\hat{\sM}}*}\tilde B: \sT_{\hat{\sM}}\cong \Omega^1_{\hat{\sM}}$. 

In fact, $R^1p_{{\hat{\sM}}*}\tilde B$ induces a symplectic form on $\hat{\sM}$ and restrict to the canonical symplectic form on $T^\vee_{\hat{\sN}}\subset \hat{\sM}$. The following lemma shows that $R^1p_{\hat{\sM},*}\tilde{B}$ is anti-symmetric.

\blemma Let $V$ be a vector space over $k$, such that $V\cong V_0\oplus V_1$ with subspaces $V_0,V_1$ of the same dimension $n$. If we have an isomorphism $\phi$ and $B_V$ such that they fit into the following diagram:
\[\begin{tikzcd}
	0\arrow[r]&V_0\arrow[r]\arrow[d,"\cong","{-(\phi^\vee)^{-1}}"']&	V \arrow[r] \arrow[d,"\cong","B_V"'] 
	& V_1 \arrow[d,"\cong","\phi"']\arrow[r]&0 \\
	0\arrow[r] &V_1^\vee\arrow[r]&	V^\vee \arrow[r]
	& V_0^\vee\arrow[r]&0
\end{tikzcd}\]
then $B_V$ induces a non-degenerate anti-symmetric bilinear form on $V$. 
\elemma

Let $Z$ be a smooth variety, the tangent sheaf of the cotangent bundle $T^\vee_Z$, has a decomposition $\sT_{T^\vee_Z}\cong \pi_Z^*\sT_Z\oplus \pi_Z^*\Omega^1_Z$. Then we have the following isomorphisms
\[\begin{tikzcd}
	0\arrow[r]&\pi_Z^*\Omega^1_Z\arrow[r]\arrow[d,"\cong","{-(\phi^\vee)^{-1}}"']&	\sT_{T^\vee_Z} \arrow[r] \arrow[d,"\cong","B_Z"'] 
	& \pi_Z^*\sT_Z \arrow[d,"\cong","\phi"']\arrow[r]&0 \\
	0\arrow[r] &(\pi_Z^*\sT_Z)^\vee\arrow[r]&	\sT_{T^\vee_Z}^\vee \arrow[r]
	& (\pi_Z^*\Omega^1_Z)^\vee\arrow[r]&0
\end{tikzcd}\]
where $\phi$ is induced by the isomorphism $\sT_Z\cong_{\sO_Z} (\Omega_Z^1)^\vee$ and $B_Z$ is given by $\phi$ and $-(\phi^\vee)^{-1}$. If we check by local coordinate, we will see that the symplectic form $B_V$ is just the canonical symplectic form on the cotangent bundle $T^\vee_Z$.

We now return to our moduli spaces case and show that the 2-form defined by $R^1p_{\hat{\sM},*}\tilde{B}$ coincides with the canonical symplectic form on $T^{\vee}_{\hat{\sN}}$, hence it is a closed symplectic form.
\bproposition\label{symplecticSL} $R^1p_{{\hat{\sM}}*}\tilde B$ induces a symplectic form on $\hat{\sM}$ and restrict to the canonical symplectic form on $T^\vee_{\hat{\sN}}\subset \hat{\sM}$. And this implies that the moduli space of stable parabolic $\SL_r$ Higgs bundles $\hat{\sM}$ has trivial canonical bundle.
\eproposition
\begin{proof}
	Let us restrict $R^1p_{\hat{\sM},*}\tilde{B}$ to the open subvariety $T^\vee_{\hat{\sN}}$. 
	
	Since one has $\sP ar\sE nd^0(\bsf U)\cong \pi_{X\times \hat{\sN}}^*\sP ar\sE nd^0(\bsf V)$ and $\pi_{\hat{\sN}}^*\sT_{\hat{\sN}}\cong R^1p_{{\hat{\sM}}*}(\sP ar\sE nd^0(\bsf U))$, then the direct image $R^1p_{{\hat{\sM}}*}$ of the truncation
	\[\sS\sP ar\sE nd^0(\bsf U)\otimes p_X^*\omega_X(D)[-1]\to \tx{ad}^\pt_{\uptheta_\bsf U}\to\sP ar\sE nd^0(\bsf U)\xrightarrow{+1} \]
	induces the short exact sequence
	\begin{footnotesize}
		\[\begin{tikzcd}
			0\arrow[r]&p_{{\hat{\sM}}*}(\sS\sP ar\sE nd^0(\bsf U)\otimes p_X^*\omega_X(D))\arrow[r]\arrow[d,"\cong"]&	R^1p_{{\hat{\sM}}*}(\tx{ad}^\pt_{\uptheta_\bsf U}) \arrow[r] \arrow[d,"\cong"] 
			& R^1p_{{\hat{\sM}}*}(\sP ar\sE nd^0(\bsf U)) \arrow[d,"\cong"]\arrow[r]&0 \\
			0\arrow[r]&\pi_{\hat{\sN}}^*\Omega^1_{\hat{\sN}}\arrow[r]&	\sT_{T^\vee_{\hat{\sN}}} \arrow[r] 
			& \pi_{\hat{\sN}}^*\sT_{\hat{\sN}}\arrow[r]&0
		\end{tikzcd}\]
	\end{footnotesize}
	which is isomorphic to the first exact sequence of the tangent sheaf of $T^\vee_{\hat{\sN}}/\hat{\sN}$.
	
	Similarly,  the direct image $R^1p_{{\hat{\sM}}*}$ of the truncation triangle of  ${\tx{ad}^\pt_{\uptheta_\bsf U}}^\vee\otimes p^*\omega_X[-1]$ is isomorphic to the dual exact sequence $0\to (\pi_{\hat{\sN}}^*\sT_{\hat{\sN}})^\vee\to (\sT_{T^\vee_{\hat{\sN}}})^\vee \to (\pi_{\hat{\sN}}^*\Omega^1_{\hat{\sN}})^\vee\to 0$.
	
	Thus $R^1p_{{\hat{\sM}}*}\tilde B: \sT_{\hat{\sM}}\cong \Omega^1_{\hat{\sM}}$ fits into the exact sequence
	\begin{tiny}
		\[\begin{tikzcd}
			0\arrow[r]& p_{{\hat{\sM}}*}(\sS\sP ar\sE nd^0(\bsf U)\otimes p_X^*\omega_X(D))\arrow[r]\arrow[d,"\cong","{R^1p_{{\hat{\sM}}*}\phi}"']&	R^1p_{{\hat{\sM}}*}(\tx{ad}^\pt_{\uptheta_\bsf U}) \arrow[r] \arrow[d,"\cong","{R^1p_{{\hat{\sM}}*}\tilde B}"'] 
			& R^1p_{{\hat{\sM}}*}(\sP ar\sE nd^0(\bsf U)) \arrow[d,"\cong","{-{R^1p_{{\hat{\sM}}*}{\phi^{\vee}}^{-1}}}"']\arrow[r]&0 \\
			0\arrow[r]& p_{\hat{\sM}_*}{\sP ar\sE nd^0(\bsf U)}^\vee\otimes p^*\omega_X\arrow[r]&	R^1p_{{\hat{\sM}}*}({\tx{ad}^\pt_{\uptheta_\bsf U}}^\vee\otimes p^*\omega_X[-1]) \arrow[r] 
			& R^1p_{\hat{\sM}*}(\sS\sP ar\sE nd^0(\bsf U)\otimes p_X^*\omega_X(D))^\vee\otimes p_X^*\omega_X\arrow[r]&0
		\end{tikzcd}\]
	\end{tiny}
\end{proof}
	
	
	
	As a result, we obtain the following which will be used in the calculation of p-adic integration.
	\begin{corollary}  The $\SL_r$ parabolic Hitchin map $\hat h:\hat {\sM}\to \sH_P^0$ is an l.c.i. map. $\omega_{\hat{\sM}}\cong \omega_{\hat{h}}\otimes \hat{h}^*\omega_{\sH^0_P}$ and the relative dualizing sheaf $\omega_{\hat{h}}$ is trivial. In particular, when restricted to generic fibers, 
	\end{corollary}
	\begin{proof}
		Since both $\hat {\sM}$ and $\sH_P^0$ are smooth, $\hat{h}$ is a l.c.i. by \cite[Lemma 37.59.11, tag 0E9K]{stacks-project}. Hence its realtive dualizing complex is concentrated in degree zero by \cite[Lemma 48.29.2, tag 0E9Z]{stacks-project}. We denote it by $\omega_{\hat{h}}$. And by \cite[Lemma 48.28.10, tag 0E30]{stacks-project}, we have $\omega_{\hat{\sM}}\cong \omega_{\hat{h}}\otimes \hat{h}^*\omega_{\sH^0_P}$. Hence the relative dualizing sheaf $\omega_{\hat{h}}$ is trivial, in particular, it restricts to the translation invariant top forms on generic fibers.
	\end{proof}
	\begin{remark} When the base field is $\mathbb C$, the symplectic structure of the parabolic Hitchin system is showed in \cite{BR94} and \cite{Botta95}.
	\end{remark}

	\section{Topological Mirror Symmetry}
	In this section, we first present some results in the seminal paper \cite{GWZ20m} without going into details. Then we prove the topological mirror symmetry for the parabolic $\operatorname{SL_r}$/$\operatorname{PGL_r}$ Hitchin systems.
		%

	\subsection{Twisted Stringy $E$-Polynomials and Point Counting}
	Let $\Gamma$ be a finite groups acting generically fixed-point freely on a smooth quasi-projective variety { $V$} of dimension $n$ over a field $k$. We assume $\#\Gamma$ is invertible in $k$. For each element $r\in\Gamma$, we put $C(\gamma)$ its centralizer in $\Gamma$. We denote by $[V/\Gamma]$ the quotient stack over $k$. In the following, we fix a system of primitive root of units satisfying that:
	\[
	\xi_{mn}^{n}=\xi_{m}
	\]
	where $\xi_{\ell}$ is the $\ell$-th primitive root of unit we choose.
	\begin{definition}\label{def:stringy polynomials}
		When $k=\BC$,
		\begin{enumerate}
			\item The stringy $E$ polynomial of $[V/\Gamma]$ is defined as:
			\[
			E_{\text{st}}([V/\Gamma];u,v)=\sum_{\gamma\in \Gamma/{\text{conj}}}(\sum_{Z\in\pi_{0}([V^{\gamma}/C(\gamma)])}E(Z;u,v)(uv)^{F(\gamma,Z)})
			\]
			where $\Gamma/{\text{conj}}$ is the set of conjugacy classes of $\Gamma$, the second summation is over connected components of $[V^{\gamma}/C(\gamma)]$. Here $F(\gamma,Z)$ is the Fermionic shift defined as follows, let $x\in X^{r}$ with image in $Z$, then $\gamma$ acts on $T_xV$ with eigenvalue $\{\xi^{c_i}\}_{i=1,\ldots,n}$ where $0\leq c_i<\#\langle\gamma\rangle$, $\#\langle\gamma\rangle$ is the order of $\gamma$,  for $i=1,\ldots,n$. Then 
			\[F(\gamma,x)=\sum \frac{c_i}{\#\langle\gamma\rangle}\]
			It is not difficult to see that $F(\gamma,x)$ is locally constant for $x\in V^{\gamma}$, hence we put it as $F(\gamma,Z)$.
			
			We write $Z=[W/C(\gamma)]$, then:
			\[
			E(Z;u,v)=\sum_{p,q,k}(-1)^{k}\dim (\gr^{W}_{p,q}H_{c}^{k}(W)^{C(\gamma)})u^pv^q.
			\] 
			\item Let $\alpha$ be a $\mu_{r}$-gerbe on $[V/\Gamma]$, we may treat it as an element in $H^{2}(V,\mu_r)^{\Gamma}$. By the transgression, $\alpha$ defines a $\mu_r$ bundle  $\sL_{\gamma}^{\alpha}$ on $[V^\gamma/C(\gamma)]$. Then we define the stringy $E$ polynomial of $[V/\Gamma]$ twisted by $\alpha$ as follows:
			\[
			E^{\alpha}_{\text{st}}([V/\Gamma];u,v)=\sum_{\gamma\in \Gamma/{\text{conj}}}(\sum_{Z\in\pi_{0}([V^{\gamma}/C(\gamma)])}E(Z,\sL_{\gamma}^{\alpha};u,v)(uv)^{F(\gamma,Z)})
			\]
			where
			\[
			E(Z,\sL_{\gamma}^{\alpha};u,v)=E^{\chi}(\sL_{\gamma}^{\alpha};u,v)
			\]
			$\chi:\mu_{r}\rightarrow\BC^{\times}$ is the standard character and $E^{\chi}$ denotes the part of the $E$-polynomials corresponding to $\chi$-isotypic component of the cohomology of the total space $H_c^*(\sL_{\gamma}^{\alpha})$.
		\end{enumerate}
		
	\end{definition}
	Analogously, we define twisted stringy point counting over a finite field $\BF_{q}$.
	\begin{definition}\label{def:stringy counting}
		When $k=\BF_q$,
		\begin{enumerate}
			\item The stringy point counting of $[V/\Gamma]$ over $\BF_q$ is defined as follows:
			\[
			\#_{\text{st}}[V/\Gamma](\BF_q)=\sum_{\gamma\in \Gamma/\text{conj}}(\sum_{Z\in\pi_{0}([V^{\gamma}/C(\gamma)])}q^{F(\gamma,Z)}\#Z(\BF_q))
			\]
			where 
			\[\#Z(\BF_q)=\sum_{z\in Z(k)_{\isom}}\frac{1}{\Aut(z)}\]
			\item Let $\alpha$ be a $\mu_r$ gerbe on $[V/\Gamma]$ and $\sL^{\alpha}_{\gamma}$ the induced $C(\gamma)$ equivariant line bundle on $X^{r}$. Then the $\alpha$-twisted stringy point counting is defined as:
			\[
			\#^{\alpha}_{\text{st}}[V/\Gamma](\BF_q)=\sum_{\gamma\in \Gamma/\text{conj}}(\sum_{Z\in\pi_{0}([X^{\gamma}/C(\gamma)])}q^{F(\gamma,Z)}\#^{\sL_{\gamma}^{\alpha}}Z(\BF_q))
			\]
			where,
			\[
			\#^{\sL_{\gamma}^{\alpha}}Z(\BF_q))=\sum_{z\in Z(k)_{\isom}}\frac{\Tr(\Fr_{z},\sL_{\gamma,z}^{\alpha})}{\Aut(z)}
			\]
		\end{enumerate}
	\end{definition}
	Via p-adic Hodge theory and Chebatarov density theorem for Galois representation of number fields, 
	\begin{theorem}\cite{GWZ20m}
		Let $R\subset \BC$ be the finitely generated $\BZ$-algebra, (containing sufficiently many roots of unit). We fix a field isomorphism $\bar{\BQ}_{\ell}\cong \BC$. Let $V_i$ be smooth over $\Spec R$ and $\Gamma_i$ be finite abelian groups acting on $V_i$, $i=1,2$. Let $\alpha_i$ be two $\mu_r$ gerbes on $[V_i/\Gamma_i]$. If for all ring homomorphism $R\rightarrow\BF_q$, we have:
		\[
		\#^{\alpha_1}_{\text{st}}([V_1/\Gamma_1]\times_{\Spec R}\Spec \BF_q)(\BF_q)= \#^{\alpha_2}_{\text{st}}([V_2/\Gamma_2]\times_{\Spec R} \Spec \BF_q)(\BF_q)
		\]
		then
		\[
		E^{\alpha_1}_{\text{st}}([V_1/\Gamma_1]\times_{\Spec R}\Spec \BC;x,y)=E^{\alpha_2}_{\text{st}}([V_2/\Gamma_2]\times_{\Spec R}\Spec \BC;x,y)
		\]
	\end{theorem}
	Hence the equality of twisted stringy Hodge numbers is translated into an equality of stringy point counting which will be calculated by p-adic integration.
	\subsection{{$p$}-adic Measures on Orbifolds over Local Fields}\label{subsec: p-adic int on orbifold}
	Let $F$ be a $p$-adic field and $\sO_{F}$ is the integral ring with $\kappa_{F}\cong\BF_{q}$ the residue field. We now assume $V$ is smooth over $\Spec\sO_F$. And $\Gamma$ is a finite abelian group, $\#\Gamma$ is invertible over $\sO_F$, acting generically  fixed-point freely on $Y$. We put $\sM:=[V/\Gamma]$ and $M=V/\!\!/\Gamma$ is the geometric quotient. We denote $\pi:V\rightarrow M$ the quotient map. Let $\Delta$ be the locus where $\Gamma$ does not act freely, and $U=V\backslash\Delta$. We shall construct a measure on 
	\[
	M(\sO_{F})^{\#}=M(\sO_{F})\cap \pi(U)(F).
	\] 
	Notice that $\pi(U)$ is smooth over $S$, thus $\pi(U)(F)$ is an $F$-analytic manifold. Since $M(\sO_{F})^{\#}$ is open in $\pi(U)(F)$, $M(\sO_{F})^{\#}$ is an compact open $F$-analytic manifold.
	
	Since $M$ is $\BQ$-Gorenstein, thus its canonical sheaf $K_M$ over $S$ is $\BQ$-Cartier. Moreover, by \cite[Lemma 7.2]{Y17}, there is a unique $\BQ$-Cartier divisor $D$ on $M$ such that $\pi^{*}(K_M+D)=K_X$. In fact, 
	\[
	D=\sum \frac{r_i-1}{r_i}\Delta_i,
	\]
	where $\Delta_i$ is a component of push forward of the ramification divisor $\Delta$. We let $\mu_{\text{orb}}$ be the measure on $M(\sO_{F})^{\#}$ associated to the pair $(M,D)$. To be more precise, let $r$ be a positive integer, such that $r(K_M+D)$ give rise to a line bundle $I$ on $M$. Then locally, we can integrate a global section of $I$  to obtain a well-defined measure on $M(\sO_{F})^{\#}$. We may extend by zero to $M(\sO_{F})$. The divisor $D$ and the measure $\mu_{\text{orb}}$ are independent of the choice of the representation of $M$.
	\begin{remark}\label{rem:integration of top form}
		
		In particular, if $\codim \Delta\geq 2$, and there is a nowhere vanishing global section $\omega$ of  $(\Omega^{\text{top}}_{\pi(U)})^{\otimes \ell}$ for some $\ell>0$, then the $\mu_{\text{orb}}$ is the integration of $|\omega|^{1/\ell}$ on $M(\sO_{F})^{\#}\subset \pi(U)(F)$. This particular kind of $\mu_{\text{orb}}$ is what we will use in our topological mirror symmetry of parabolic Hitchin systems.
	\end{remark}
	Now given a $\mu_r$ gerbe $\alpha$ on $[V/\Gamma]$, we obtain a function (via the Hasse invariant):
	\[
	\inv_{\alpha}: M(\sO_{F})^{\#}\rightarrow\Br(F)[r]\cong\frac{\BQ}{\BZ}[r]
	\]
	and we define:
	\[
	f_{\alpha}=\exp(2\pi i\inv_{\alpha})
	\]
	we can check that $f_{\alpha}$ is locally constant and hence integrable on $M(\sO_{F})^{\#}$.
	\begin{theorem}\label{thm:twisted counting equals integration}\cite[Corollary 5.29]{GWZ20m}
		$\frac{\#_{\text{st}}^{\alpha}\sM(\kappa_F)}{q^{\dim\sM}}=\int_{M(\sO_{F})^{\#}}f_{\alpha}d\mu_{\text{orb}}$
	\end{theorem}
	As we can see from the Definition \ref{def:stringy counting}, the effect of a gerbe $\alpha$ on stringy point counting coming from the trace of local geometric Frobenius on a $\ell$-adic sheaf induced by the transgression of the gerbe $\alpha$. When translated to $p$-adic integration, the twist is given by the integrable function $f_{\alpha}$.
	\subsection{Main Theorem}
	As in Section 3, we put $\hat{\sM}=\sM_{\SL_r,P}^{\sL}, \check{\sM}=\sM_{\PGL_r,P}^{e}$,  and $\hat{M},\check{M}$ are corresponding coarse moduli space, $M(\sO_F)^{\#}$ are notations as in the Subsection \ref{subsec: p-adic int on orbifold}.
	
	We first prove the following lemma, which implies that our (orbifold) measures are naturally given by symplectic forms by Remark \ref{rem:integration of top form}.
	\begin{lemma}\label{lem:fixed point}
		Considering the natural action of $\Pic^0(X)[r]$ on $\hat{\sM}$ via tensor product,  without abuse of notation, we put $\Delta$ as the non-free subvariety of the action. Then
		\[
		\codim \Delta\ge 2
		\]
	\end{lemma}
	\begin{proof}
		Let $(E,\theta)$ be a closed point in $\Delta$. By definition, there is a torsion line bundle $\sW\in \Pic^0(X)[r]$ and an isomorphism of vector bundles:
		\[
		\phi:E\xrightarrow{\cong} E\otimes\sW
		\]
		We can treat $(E,\phi)$ as a Higgs bundle twisted by the torsion line bundle $\sW$. Hence all such $E$ are line bundles on a spectral curve living in the total space of $\sM$. Since $\sW$ is torsion, such a spectral curve is a disjoint union of $X$. Hence we know that 
		\[
		\dim\Delta\le g(X)+\dim H^0(X,\sS \sP ar\sE nd(\sE)\otimes_{\sO_X}\omega_X(D))
		\]
		In particular, 
		\[
		\codim \Delta\ge 2.
		\]
	\end{proof}
	\begin{remark}\label{rem:translation-invariant measures}
		By Remark \ref{rem:integration of top form}, this codimension 2 Lemma, combined with top forms construced from the symplectic structure in Section \ref{sec 3}, implies that we have volume forms on $\sM_{\SL_{r},P}^{\sL},\sM_{\PGL_r,P}^{e}$ whose restriction to fibers over $\sH_{P}^{0\diamond}$ are cannonical translation-invariant volumes forms on those torsors.
	\end{remark}
	Even though, we do not have the codimension 2 condition in the definition of abstract dual Hitchin system \cite[Definition 6.9]{GWZ20m}. This is compensated by the existence of symplectic forms on $\hat{\sM}=\mathcal{M}_{\SL_r,P}^{\sL}$.
	Now let us state our theorem:
	\begin{theorem}\label{main theorem}
		There is a topological mirror symmetry for parabolic Hitchin systems:
		\begin{equation}
			E_{\text{st}}^{e\hat{\alpha}}(\hat{\sM};u,v)=E_{\text{st}}^{d\check{\alpha}}(\check{\sM};u,v)
		\end{equation}
		where $d=\deg\sL$ and we require that there exists a nonzero number $\lambda$ such that $e\equiv \lambda d\ (\emph{mod}\  \Delta_P)$
	\end{theorem}
	Before we give the proof, let us first collect all the properties we have:
	\begin{enumerate}
		\item A pair of Hitchin systems:
		\[
		\begin{tikzcd}
			\hat{\sM}\supset\hat{\sM}^{\diamond}\ar[rd,"\hat{h}_P"]&&\check{\sM}\supset\check{\sM}^{\diamond}\ar[ld,"\check{h}_P"']\\
			&\sH^0_P\supset\sH^{0\diamond}_{P}
		\end{tikzcd}.
		\]
		\item Arithmetic Duality, see Theorem \ref{thm:syz mirror symmetry}.
		\item The condition:
		\begin{center}
			There exists a nonzero integer $\lambda$ such that $e\equiv \lambda d\mod \Delta_P$.
		\end{center}
		along with Corollary \ref{cor:k-rational point}, implies that we have the following:
		\begin{center}
			$\hat{h}_P^{-1}(b)(F)$ and $\check{h}_P^{-1}(b)(F)$ are both non-empty if and only if both $e\hat{\alpha}, d\check{\alpha}$ splits. 
		\end{center}
		\item There is global nowhere vanishing top forms on $\sM_{\SL_r,P}^{\sW}$ for all line bundle $\sW$ on $X$. In particular, by ``the codimension 2" Lemma \ref{lem:fixed point}, such kind of volume forms define volume forms on $\hat{\sM},\check{\sM}$. And when restricted to fibers over $\sH^{0\diamond}_{P}$, they are translation-invariant volume form on those torsors. See also Remark \ref{rem:translation-invariant measures}.
	\end{enumerate}
	
	Now we are ready to apply the p-adic formalism in \cite{GWZ20m}. For fiberwise integration, readers may refer to the calculation in \cite[Theorem 6.17]{GWZ20m}
	\begin{proof}
		
		By Theorem \ref{thm:twisted counting equals integration}, we only need to prove that:
		\begin{equation}\label{eq: integration equality}
			\int_{\hat{M}(\sO_{F})^{\#}}f_{\hat{\alpha}}d\hat{\mu}_{\text{orb}}=\int_{\check{M}(\sO_{F})^{\#}}f_{\check{\alpha}}d\check{\mu}_{\text{orb}}
		\end{equation}
		We first define:
		\[
		\sH_{P}^{0}(\sO_F)^{\flat}:=\{b\in \sH_{P}(\sO_{F})|b\in \sH_{P}^{\diamond}(F)\}.
		\]
		Roughly speaking, we want to consider those $\sO_{F}$ points of the Hitchin base $\sH_{P}^{0}$ which generic part lies in $\sH_{P}^{\diamond}$, over which the Hitchin fibers are torsors over dual Prym varieties.
		
		The next step is to calculate the integration over the whose space $\hat{\sM}$, $\check{\sM}$ via Fubini type theorem by first calculating integration over fibers. To do it,
		we define:
		\[
		\hat{M}(\sO_{F})^{\#}_a:=\{x\in \hat{M}(\sO_{F})^{\#}|\pi(x)\equiv a\ (\text{mod}\  \mathfrak{m}_F)\}
		\]
		and 
		\[
		\sH_{P}(\sO_F)^{\flat}_a:=\{b\in\sH^0_P(\sO_{F})^{\flat}|b\equiv a\ (\text{mod}\  \mathfrak{m}_F)\}
		\]
		Notice that we have:
		\[
		\hat{M}(\sO_{F})^{\#}_a\backslash \hat{h}_{P}^{-1}(\sH_P(\sO_F)^{\flat}_a)=\{x\in \hat{M}(\sO_{F})^{\#}_a|x\in \hat{h}_P^{-1}(\sH^0_P\backslash\sH_P^{\diamond})(\sO_{F})\}
		\]
		which is a closed subset in $\hat{M}(\sO_{F})^{\#}_a$ for any $a\in\sH^0_P(\kappa_F)$. And same definitions and properties for $\check{\sM}(\sO_{F})_{a}^{\#}$.
		
		Notice that, $\hat{\sM}$ is symplectic and hence admits a nowhere vanishing global top form $\hat{\omega}$. Since the parabolic Hitchin map is a local complete intersection, then the relative dualizing sheaf $\omega_{\hat{h}_{P}}$ is also a line bundle and we have:
		\begin{equation}\label{eq:composed top form}
			\hat{\omega}=\hat{h}_{P}^*\omega_{\sH_{P}}\wedge\omega_{\hat{h}_{P}}
		\end{equation}
		where $\omega_{\sH_P}$ is a nonwhere vanishing top form on the affine space $\sH_P$. Notice that, the restriction of $\omega_{\hat{h}_{P}}$ to fibers over $\sH_{P}^{\diamond}$ is the translation invariant volume forms on these torsors. 
		
		By Lemma \ref{lem:fixed point}, the subvariety of $\hat{\sM}$ where the action of $\Pic^{0}(X)[r]$ is not free is of codimension $\ge 2$. Hence by Remark \ref{rem:integration of top form}, we know that the volume form on $\hat{M}(\sO_{F})^{\#}$ resp. $\check{\sM}(\sO_F)^{\#}$ is the integration of $\hat{\omega}$ resp. $\check{\omega}$.
		
		Recall that we have:
		\[
		\hat{M}(\sO_{F})^{\#}_a\backslash \hat{h}_{P}^{-1}(\sH_P(\sO_F)^{\flat}_a)=\{x\in \hat{M}(\sO_{F})^{\#}_a|x\in \hat{h}_{P}^{-1}(\sH_P\backslash\sH_P^{\diamond})(\sO_{F})\}
		\]
		which is a closed subset in $\hat{M}(\sO_{F})^{\#}_a$ for any $a\in\sH^{0}_P(\kappa_F)$. Hence we only need to integrate over $\pi^{-1}(\sH_P(\sO_F)^{\flat}_a)$.  And this is similar for $\check{\sM}$ side.
		
		Now we may rewrite \eqref{eq: integration equality} as follows:
		
		\begin{align}\label{eq:fubini}
			&\int_{\hat{M}(\sO_{F})^{\#}}f_{\hat{\alpha}}d\hat{\mu}_{\text{orb}}\nonumber\\
			=&\sum_{a\in\sH_P(\kappa_F)}\int_{\hat{M}(\sO_{F})_a^{\#}}f_{\hat{\alpha}}d\hat{\mu}_{\text{orb}}\\
			=&\sum_{a\in\sH_P(\kappa_F)}\int_{\hat{h}_{P}^{-1}(\sH_{P}(\sO_{F})^{\flat}_a})f_{\hat{\alpha}}d\hat{\mu}_{\text{orb}}\nonumber\\
			=&\sum_{a\in\sH_P(\kappa_F)}\int_{\sH_{P}(\sO_{F})^{\flat}_a}d\mu_{\sH_{P}}\int_{\hat{h}_{P}^{-1}(b)(F)}f_{\hat{\alpha}}\abs{\omega_{\hat{h}_{P}}}\nonumber
		\end{align}

		Similarly, we have:
		\begin{align*}
			\int_{\check{M}(\sO_{F})^{\#}}f_{\check{\alpha}}d\check{\mu}_{\text{orb}}
			=\sum_{a\in\sH_P(\kappa_F)}\int_{\sH_{P}(\sO_{F})^{\flat}_a}d\mu_{\sH_{P}}\int_{\check{h}_{P}^{-1}(b)(F)}f_{\check{\alpha}}\abs{\omega_{\check{h}_{P}}}.
		\end{align*}

		Then by \cite[Theorem 6.17]{GWZ20m}, for all $a\in\sH_P^0(\kappa_F)$ and each $b\in \sH_{P}(\sO_F)_{a}$we have the fiberwise equality:
		\begin{align*}
			\int_{\hat{h}_{P}^{-1}(b)(F)}f_{\hat{\alpha}}\abs{\omega_{\hat{h}_{P}}}
			=\int_{\check{h}_{P}^{-1}(b)(F)}f_{\check{\alpha}}\abs{\omega_{\check{h}_{P}}}.
		\end{align*}
		Hence:
		\[
		\int_{\hat{M}(\sO_{F})^{\#}}f_{\hat{\alpha}}d\hat{\mu}_{\text{orb}}=\int_{\check{M}(\sO_{F})^{\#}}f_{\check{\alpha}}d\check{\mu}_{\text{orb}}
		\]
		And their twisted stringy Hodge polynomials are equal.
	\end{proof}
	We can give a generalization of \cite[Theorem 3.13]{Gothen19} which prove the topological mirror symmetry without the gerbe-twist in rank 2,3 cases. 
	\begin{theorem}\label{thm:delta_P is one}
		If $\Delta_P=1$, then for all integer $d,e$, we have:
		\[
		E_{\text{st}}^{e\hat{\alpha}}(\hat{\sM};u,v)=E_{\text{st}}^{d\check{\alpha}}(\check{\sM};u,v)=E_{\text{st}}(\hat{\sM};u,v)=E_{\text{st}}(\check{\sM};u,v)
		\]
		In particular, $\Delta_P=1$ holds if there exists a marked $x$ such that $P_x$ is Borel, i.e., a full flag filtration at $x$.
	\end{theorem}
	
	\begin{proof}
		For any $d,e$, the equality:
		\[
		E_{\text{st}}^{e\hat{\alpha}}(\hat{\sM};u,v)=E_{\text{st}}^{d\check{\alpha}}(\check{\sM};u,v)
		\]
		follows directly from Theorem \ref{main theorem}.
		
		To prove the remaining equalities, we first point out the all these (twisted) stringy $E$-polynomials only depends on $\deg\sL,e$.
		
		By Corollary \ref{cor:deg 1 div}, we have a $k$ rational point in $\Pic^{d}(\tilde{X}_{a})$. Hence we can choose $\sW\in \Pic^{d}(X)(k),\sW'\in\Pic^{e}(X)(k)$ such that $\Nm^{-1}(\sW)(k)$ is non-empty. We apply this to the function field of $\sH_{P}$.  Hence then there are sections\footnote{Notice that in the following expressions, there is an abuse notations. Here $\hat{\sM}=\sM_{\SL_r,P}^{\sW},\check{\sM}=\sM_{\PGL_r,P}^{\deg\sW}$}:
		\[
		\sH_{P}^{0\diamond}\rightarrow \hat{\sM}^{\diamond}, \sH_{P}^{0\diamond}\rightarrow \check{\sM}^{\diamond}
		\]
		which means that for all $a\in\sH^{0}_{P}(\kappa_F)$ and all $b\in \sH_{P}(\sO_F)^{\flat}_a$, $\hat{h}_P^{-1}(b)(F)$ and $\check{h}_{P}^{-1}(b)(F)$ are non-empty. In particular, this means that $f_{\hat{\alpha}}=1$ over all $\hat{h}_P^{-1}(b)(F)$. Hence from Equation \eqref{eq:fubini}:
		\[
		\int_{\hat{M}(\sO_{F})^{\#}}f_{\hat{\alpha}}d\hat{\mu}_{\text{orb}}=\int_{\hat{M}(\sO_{F})^{\#}}d\hat{\mu}_{\text{orb}}
		\]
		which means that $E_{\text{st}}^{e\hat{\alpha}}(\hat{\sM};u,v)=E_{\text{st}}(\hat{\sM};u,v)$. And similarly for $\PGL_r$ side.
	\end{proof}

\end{document}